\documentclass[10pt]{article}
\pdfoutput=1
\usepackage{amsmath, amsfonts, amsthm, amssymb, hyperref, setspace, geometry, mathrsfs}
\title{\textbf{On the $ICPC$-property of finite subgroups}\thanks{\footnotesize  \scriptsize\emph{E-mail addresses:}
      zsmcau@cau.edu.cn\,(S. Zhang).}}
 
\author{Shengmin Zhang\\
\quad
\\
{\small College of Science,
China Agricultural University,
Beijing 100083, China}}
\date{}
\newtheorem{theorem}{Theorem}[section]

\newtheorem{lemma}[theorem]{Lemma}
\newtheorem{corollary}[theorem]{Corollary}
\theoremstyle{definition}
\newtheorem{definition}[theorem]{Definition}

\allowdisplaybreaks

\expandafter\let\expandafter\oldproof\csname\string\proof\endcsname
\let\oldendproof\endproof
\renewenvironment{proof}[1][\proofname]{%
  \oldproof[\bfseries\scshape #1]%
}{\oldendproof}

\usepackage{stackengine,scalerel}
\stackMath
\def\trianglelefteqslant{\ThisStyle{\mathrel{%
  \stackinset{r}{.75pt+.15\LMpt}{t}{.1\LMpt}{\rule{.3pt}{1.1\LMex+.2ex}}{\SavedStyle\leqslant}%
}}}
\renewcommand{\unlhd}{\trianglelefteqslant}

\renewcommand{\leq}{\leqslant}
\renewcommand{\geq}{\geqslant}
\begin{document}
\maketitle
\begin{abstract}
Let $G$ be a finite group and $A$ be a subgroup of $G$. Then $A$ is called a $p$-$CAP$-subgroup of $G$, if $A$ covers or avoids every $pd$-chief factor of $G$. A subgroup $H$ of $G$ is said to be an $ICPC$-subgroup of $G$, if $H \cap [H,G] \leq H_{pcG}$, where $H_{pcG}$ is a $p$-$CAP$-subgroup of $G$ contained in $H$. In this paper, we investigate the structure of $G$ under the assumption that certain subgroups are $ICPC$-subgroups of $G$, and characterization of $p$-nilpotency and other results are obtained.
\end{abstract}
\section{Introduction}
All groups in this paper are finite, and $G$ always denotes a group. Let $H$ be a subgroup of $G$. Then $[H,G]$ denotes the commutator group of $H$ and $G$, and $H^G$ denotes the normal closure of $G$ in $H$, i.e. the smallest normal subgroup of $G$ which contains $H$. It is obvious that $[H,G]$ is normal in $G$, and the product of $H$ and $[H,G]$ is a group, with the fact that $H [H,G] =H^G$. Since $H^G$ and $[H,G]$ are closely related to the structure of $G$, it is interesting to consider the intersection of $H$ and $[H,G]$. In fact, we can restrict the intersection of $H$ and $[H,G]$ into certain subgroups of $H$. For instance, in {{\cite{GL}}}, Y. Gao and X. Li considered the $IC \Phi$-subgroups of $G$. A subgroup $H$ of $G$ is called an $IC \Phi$-subgroup of $G$, if $H \cap [H,G] \leq \Phi (H)$. In their paper, some characterizations of $p$-nilpotency and supersolvability were obtained under the assumption that certain subgroups of $G$ are $IC \Phi$-subgroups of $G$. J. Kaspczyk  studied in {{\cite{KA}}} the properties of $IC \Phi$-subgroups as well, and concluded characterization of $p$-nilpotency under the assumption that certain subgroups of $G$ of fixed order are $IC \Phi$-subgroups of $G$. In 2022, Y. Gao and X. Li {{\cite{GL2}}} generalised the concept of $IC \Phi$-subgroup and introduced the following definition: A subgroup $H$ of $G$ is said to be an $IC \Phi s$-subgroup of $G$, if $(H \cap [H,G])H_G /H_G \leq \Phi (H/H_G) H_{sG}/H_G$, where $H_G$ denotes the largest normal subgroup of $G$ contained in $H$, and $H_{sG}$ denotes the unique maximal subgroup of $H$ which is $s$-permutable in $G$. Some characterizations of $p$-nilpotency and solvably saturated formation containing $\mathfrak{U}$ were obtained. Recall that a saturated formation $\mathfrak{F}$ is said to be solvably saturated, if $G \in \mathfrak{F}$ whenever $G / \Phi^{*} (G) \in \mathfrak{F}$. This concept was first introduced by W. Guo and A. N. Skiba {{\cite{GS2}}}. In {{\cite{KA2}}}, J. Kaspczyk considered further properties of $IC \Phi$-subgroups, and obtained some more generalised results.

In 2023, Y. Gao and X. Li combined the definition of $IC$ property and semi $CAP$-subgroup {{\cite{GL3}}}. Recall that a subgroup $H$ is called a semi-$CAP$-subgroup of $G$, if there exists a chief series $\Gamma$ of $G$ such that $H$ covers or avoids every chief factor in $\Gamma$. In {{\cite{GL3}}}, they introduced the following definition.
\begin{definition} 
$A$ subgroup $H$ of a group $G$ is called an $ICSC$-subgroup of $G$, if $H \cap [H, G] \leq  H_{scG}$, where $H_{scG}$ is a semi-$CAP$-subgroup of $G$ contained in $H$.
\end{definition}
In their paper, they gave the following theorems.
\begin{theorem}[{{\cite[Theorem 3.1]{GL3}}}]\label{00001}
Let $G$ be a group and $P$ a normal $p$-subgroup of $G$. Assume that every cyclic subgroup of $P$ of order $p$ and $4$ (if $P$ is a non-abelian $2$-group) either is an $ICSC$-subgroup of $G$ or has a supersolvable supplement in $G$. Then $P \leq Z_{\mathfrak{U}} (G)$. 
\end{theorem}
\begin{theorem}[{{\cite[Theorem 3.2]{GL3}}}]\label{00002}
Let $E$ be a normal subgroup of $G$ and $P$ a Sylow $p$-subgroup of $E$ with $p = \min \pi(E)$. Assume that every cyclic subgroup of $P$ of order $p$ and $4$ (if $P$ is a non-abelian $2$-group) is an $ICSC$-subgroup of $G$, then $E$ is $p$-nilpotent.
\end{theorem}
The $CAP$-property has many generalisations. We can change the chief factors the subgroups cover or avoid, and get different generalisations of $CAP$-property. Z. Gao {\it et al.} in {{\cite{GQ}}} introduced a version of generalisation of $CAP$-property as follows.
\begin{definition} 
A subgroup $A$ of a group $G$ is called a $p$-$CAP$-subgroup of $G$ if $A$ covers or avoids every $pd$-chief factor of $G$, where $pd$-chief factor denotes the chief factor of $G$ with order divided by $p$.
\end{definition}
Then a natural way to generalise the $ICPC$-subgroup comes to our mind.   Actually, we generalise the definition of $ICPC$-subgroup by changing the restriction of the intersection of $H$ and $[H,G]$ from $H_{scG}$ into $H_{pcG}$, where $H_{pcG}$ denotes a $p$-$CAP$-subgroup of $G$ contained in $H$. Thus we have the following definition.
\begin{definition}
A subgroup $A$ of $G$ is said to be an $ICPC$-subgroup of $G$, if $A \cap [A,G] \leq A_{pcG}$.
\end{definition}
In this paper, we investigate the basic properties of $ICPC$-subgroup, and obtained some characterizations of solvably saturated formation and  $p$-nilpotency under the assumption that certain subgroups of $G$ are $ICPC$-subgroups of $G$ as generalisations of Theorem \ref{00001} and Theorem \ref{00002}. Our main results are listed as follows.
\begin{theorem}\label{3.1}
Let $G$ be a finite group and $P$ be a normal $p$-group of $G$. Suppose that any cyclic subgroup of $P$ of order $p$ and $4$ (if $P$ is a non-abelian $2$-group) either is an $ICPC$-subgroup of $G$ or has a supersolvable supplement in $G$. Then $P \leq Z_{\mathfrak{U}} (G)$. 
\end{theorem}
\begin{corollary}\label{10001}
Let $\mathfrak{F}$ be a solvably saturated formation containing $\mathfrak{U}$ and $N$ a solvable normal subgroup of $G$ such that $G/N \in \mathfrak{F}$. Assume that every cyclic subgroup of every non-cyclic Sylow subgroup $P$ of $F(N)$ of prime order or $4$ (if $P$ is a non-abelian $2$-group) either is an $ICPC$-subgroup of $G$ or has a supersolvable supplement in $G$. Then $G \in \mathfrak{F}$.
\end{corollary}
\begin{theorem}\label{10002}
Let $G$ be a finite group, $E$ be a normal subgroup of $G$, and $P$ be a Sylow $p$-subgroup of $E$, where $p = \min \pi(E)$. Assume that every cyclic subgroup of $P$ of order $p$ and $4$ (if $P$ is a non-abelian $2$-group) is an $ICPC$-subgroup of $G$, then $E$ is $p$-nilpotent.
\end{theorem}
\section{Preliminaries}
\begin{lemma}[{{\cite[Lemma 2.1]{GQ}}}]\label{1}
Let $H$ be a $p$-$CAP$ subgroup of $G$ and $N \unlhd G$. Then the following statements are true.
\begin{itemize}
\item[(1)] $N$ is a $p$-$CAP$ subgroup of $G$.
\item[(2)] If $H \geq N$, then $H/N$ is a $p$-$CAP$ subgroup of $G/N$.
\item[(3)] If $(|H|,|N|)=1$, then $HN/N$ is a $p$-$CAP$ subgroup of $G/N$.
\end{itemize}
\end{lemma}
\begin{lemma}
Let $G$ be a finite group and $N$ a normal subgroup of $G$. Suppose that $H$ is an $ICPC$-subgroup of $G$, and $(|H|,|N|) = 1$. Then $H N/N$ is an $ICPC$-subgroup of $G/N$.
\end{lemma}
\begin{proof}
Since $(|H|,|N|)=1$, it follows that
$$|HN \cap [H,G]| = \frac{|HN||[H,G]|}{|HN[H,G]|}=\frac{|N \cap H[H,G]||H||[H,G]|}{|H[H,G]|}=|H \cap [H,G]||N \cap H[H,G]|. $$
Hence we have that $HN \cap [H,G] = (H \cap [H, G])(N \cap H[H, G])$, and:
$$H N \cap [H N, G]N = H N \cap [H, G]N = (H N \cap [H, G])N = (H \cap [H, G])N. $$
Thus we conclude that
$$H N/N \cap [H N/N, G/N] = (H N \cap [H N, G]N)/N = (H \cap [H, G])N/N \leq  H_{pcG} N/N. $$
Then lemma \ref{1} (3) yields that $H_{pcG} N/N$ is a $p$-$CAP$ subgroup of $G/N$ which is contained in $HN/N$. Therefore we get that $H N/N \cap [H N/N, G/N] \leq (HN/N)_{pc(G/N)}$. In other words, $H N/N$ is an $ICPC$-subgroup of $G/N$.
\end{proof}
\begin{lemma}[{{\cite[Lemma 4.3]{GS}}}]\label{3}
Let $P$ be a finite $p$-group, and $C$ be a Thompson critical subgroup of $P$. Then the group $D:= \Omega (C)$ has exponent $p$ or $4$ (if $P$ is a non-abelian $2$-group). 
\end{lemma}
\begin{lemma}[{{\cite[Lemma 2.12]{CG}}}]\label{4}
Let $\mathfrak{F}$ be a solvably saturated formation, $P$ be a normal $p$-subgroup of $G$, and $C$ be a Thompson critical subgroup of $P$. Suppose that either $\Omega (C) \leq Z_{\mathfrak{F}} (G)$ or $P/\Phi (P) \leq Z_{\mathfrak{F}} (G/\Phi(P))$ holds, then $P \leq  Z_{\mathfrak{F}} (G)$. 
\end{lemma}
\begin{lemma}[{{\cite[Theorem B]{SK}}}]\label{5}
Let $\mathfrak{F}$ be any formation and $G$ a group. If $E \unlhd G$  and $F^{*} (E) \leq Z_{\mathfrak{F}} (G)$, then $E \leq Z_{\mathfrak{F}} (G)$.
\end{lemma}
\begin{lemma}[{{\cite[Lemma 3.3]{GS2}}}]\label{6}
Let $\mathfrak{F}$ be a solvably saturated formation containing $\mathfrak{U}$ and $N$ a normal subgroup of $G$ such that $G/N \in  \mathfrak{F}$. If $N \leq Z_{\mathfrak{F}} (G)$, then $G\in  \mathfrak{F}$. 
\end{lemma}
\begin{lemma}[{{\cite[Lemma 2.8]{WW}}}]\label{7}
Let $M$ be a maximal subgroup of $G$ and $Q$ a normal $p$-subgroup of
$G$ such that $G = MQ$, where $p$ is a prime. Then $Q \cap M \unlhd G$.
\end{lemma}
\section{Proofs of the main theorems} 
\begin{proof}[Proof of Theorem \ref{3.1}]
Suppose that the theorem is false, and let $(G,P)$ be a pair of counterexample such that $|G|+|P|$ is of minimal order. Then clearly $G \neq Z_{\mathfrak{U}} (G)$. Let $K \unlhd G$ such that $P/K$ is a chief factor of $G$. Then the pair $(G,K)$ satisfies the hypothesis, and so $K \leq Z_{\mathfrak{U}} (G)$ by the minimal choice of $G$. Let $L$ be any normal subgroup of $G$ such that $L <P$. By our argument above, we have that $L \leq Z_{\mathfrak{U}} (G)$. If $L \not\leq K$, then $KL =P$, which implies that $P \leq Z_{\mathfrak{U}} (G)$, a contradiction. Hence we get that $L \leq K$, and so $K$ is the unique normal subgroup of $G$ such that $P/K$ is a chief factor of $G$. If $|P/K| =p$, it follows that $P/K$ is $\mathfrak{U}$-central in $G$. By generalised Jordan-Holder Theorem, every chief factor below $P$ is $\mathfrak{U}$-central in $G$. Thus we conclude that $P \leq Z_{\mathfrak{U}} (G)$, a contradiction. Hence we have $|P/K| >p$. Now let $C$ be a Thompson critical subgroup of $P$. If $\Omega (C) <P$, it yields that $\Omega (C) \leq K \leq Z_{\mathfrak{U}} (G)$. By Lemma \ref{4}, we conclude that $P \leq Z_{\mathfrak{U}} (G)$, a contradiction. Therefore we get that $\Omega (C) = P$, which indicates that $exp(P) =p$ or $4$ (if $P$ is a non-abelian $2$-group) by Lemma \ref{3}. Let $G_p$ be a Sylow $p$-subgroup of $G$ containing $P$. It follows from {{\cite[Chapter III, 3.1.11]{KS}}} that $P/K \cap Z(G_p /K) \neq 1$. Now we may assume that $R/K \leq P/K \cap Z(G_p /K)$ with $|R/K| =p$. Let $x \in R\setminus K$. Then we get that $R = K \langle x \rangle$, and $o(x) = p$ or $4$ by $R \leq \Omega (C) = P$.  

Suppose firstly that $\langle x \rangle$ is an $ICPC$-subgroup of $G$. By our hypothesis, we assert that $\langle x \rangle \cap [ \langle x \rangle,G] \leq \langle x \rangle_{pcG}$. Assume that $[\langle x \rangle ,G] < P$, it follows from $[\langle x \rangle ,G] \unlhd G$ that $[\langle x \rangle ,G] \leq K$. Thus we have that $R = \langle x \rangle K =[\langle x \rangle ,G]\cdot  \langle x \rangle K = \langle x \rangle ^G K \unlhd G$. Hence we get that $R = P$, which indicates that $|P/K| = |R/K| = p$, a contradiction. Therefore we assert that $[\langle x \rangle,G] =P$, which implies that $\langle x \rangle \cap [ \langle x \rangle,G]=\langle x \rangle = \langle x \rangle_{pcG}$. Since $p \,|\,|P/K|$, we get that $\langle x \rangle$ covers or avoids $P/K$. If $\langle x \rangle$ covers $P/K$, it follows directly that $P = P \langle x \rangle = K \langle x \rangle = R$. Hence we get that $|P/K| = |R/K| =p$, a contradiction. Thus we have that $\langle x \rangle \cap K = \langle x \rangle \cap P$, which yields that $\langle x \rangle \leq K$, a contradiction to the choice of $x$. 

Hence we conclude that there exists a supersolvable supplement $A$ in $G$ such that $G = \langle x \rangle A = PA$. Since $G \neq  Z_{\mathfrak{U}} (G)$, it follows that $A$ is a proper subgroup of $G$. Thus there exists a maximal subgroup $M$ of $G$ such that $A \leq M$. Hence we get that $G =PM$. By Lemma \ref{7}, $P \cap M \unlhd G$. Clearly $P \cap M <M$, thus we have $P \cap M \leq K$. It yields that
$$P =P \cap \langle x \rangle A = \langle x \rangle (P \cap A) \leq\langle x \rangle (P \cap M) \leq \langle x \rangle K =R.  $$ 
Hence $P=R$, which implies that $|P/K| = |R/K| =p$, a contradiction. Therefore no such counterexample of $(G,P)$ exists and we are done. 
\end{proof}

\begin{proof}[Proof of Corollary \ref{10001}]
Let $P$ be a Sylow $p$-subgroup of $F(N)$, where $p \in \pi (F(N))$. It follows directly that $P \unlhd G$. Suppose that $P$ is cyclic, then every subgroup of $P$ is characteristic in $P$, hence normal in $G$. Let $K/L$ be a chief factor of $G$ below $P$. It yields that $K/L$ is of order $p$ since every subgroup of $P$ is normal in $G$. Therefore we have $P \leq Z_{\mathfrak{U}} (G)$. Suppose that $P$ is not cyclic. Then every cyclic subgroup of $P$ of order $p$ or $4$ (if $P$ is a non-abelian $2$-group) is an $ICPC$-subgroup of $G$. By Theorem \ref{3.1},  we conclude that $P \leq Z_{\mathfrak{U}} (G)$. Thus we get that $F(N) \leq  Z_{\mathfrak{U}} (G)$, and so $F^{*} (N) \leq  Z_{\mathfrak{U}} (G)$ since $F^{*} (N) = F(N)$ by the solvability of $N$. By Lemma \ref{5}, we have that $N \leq Z_{\mathfrak{U}} (G)$. By Lemma \ref{6}, we conclude that $G \in \mathfrak{U} \subseteq \mathfrak{F}$ and we are done.
\end{proof}

\begin{proof}[Proof of Theorem \ref{10002}]
Suppose that the theorem is false, and let $(G,E)$ be a pair of counterexample such that $|G| +|E|$ is minimal. By the Burnside's Theorem, $|P| \geq p^2$. 

Now we claim that $O_{p'} (E) =1$. Assume that $O_{p'} (E)>1$, then $E /O_{p'} (E)$ is a normal subgroup of $G/O_{p'} (E)$ since $O_{p'} (E) \unlhd G$. It is clear that $P O_{p'} (E)/O_{p'} (E)$ is a Sylow $p$-subgroup of $G/O_{p'} (E)$, where $p = \min \pi (E/O_{p'} (E))$. Since $(|O_{p'} (E)|,p)=1$, every cyclic subgroup of $P O_{p'} (E)/O_{p'} (E)$ of order $p$ or $4$ (if $P$ is a non-abelian $2$-group) can be written as $\langle x \rangle O_{p'} (E)/O_{p'} (E)$, where $o(x) =p$ or $4$ (if $P$ is a non-abelian $2$-group), and $\langle x \rangle \leq P$ by Sylow Theorem. By our hypothesis, $\langle x \rangle$ is an $ICPC$-subgroup of $G$ for any such $x$. Again, we assert by Lemma \ref{1} (3) and $(|O_{p'} (E)|,p)=1$ that $\langle x \rangle O_{p'} (E) /O_{p'} (E)$ is an $ICPC$-subgroup of $G/O_{p'} (E)$. Hence the pair $(G/O_{p'} (E),E/O_{p'} (E))$ satisfies the hypothesis, and we conclude from the minimal choice of $(G,E)$ and $O_{p'} (E)>1$ that $E/ O_{p'} (E)$ is $p$-nilpotent. Therefore, $E$ is $p$-nilpotent, a contradiction. Hence we get that $O_{p'} (E) =1$.

Assume firstly that $O_p (E) >1$. It follows from $O_p (E) \leq P$ that all cyclic subgroups of $O_p (E)$ of order $p$ and $4$ (if $P$ is a non-abelian $2$-group) are $ICPC$-subgroups of $G$. Hence the pair $(G,O_p (E))$ satisfies the hypothesis of  Theorem \ref{3.1}, which indicates that $O_p (E) \leq Z_{\mathfrak{U}} (G)$. Moreover, $O_p(E) \leq Z_{\infty}(E)$ and so $O_p(E) < P$. Now let $L$ be a normal subgroup of $G$ such that $O_p (E) <L \leq E$, where $L/O_p (E)$ is a chief factor of $G$. Assume that $p \nmid |L/O_p (E)|$, then $O_p (E) \in {\rm Syl}_p (L)$, which implies that $L <E$. Since $L \cap P$ is a Sylow $p$-subgroup of $L$, it follows from every cyclic subgroup of $L \cap P$ of order $p$ or $4$ (if $P$ is a non-abelian $2$-group) is an $ICPC$-subgroup of $G$ that the pair $(G,L)$ satisfies the hypothesis. Hence $L $ is $p$-nilpotent by the minimal choice of $(G,E)$. Since $O_{p'} (E)=1$, it follows that $L$ is a $p$-subgroup, which yields that $L \leq  O_p (E)$, a contradiction to $L > O_p (E)$. Therefore we get that $p \,|\,|L/O_p (E)|$. Assume that there exists a normal subgroup $R$ of $G$  such that $R \neq O_p (E)$, and $L/R$ is a chief factor of $G$. Using the same method above, we conclude that $(G,R)$ satisfies our hypothesis. Thus by the minimal choice of $(G,E)$, we have that $R$ is $p$-nilpotent. It follows directly from $O_{p'} (E) =1$ that $R \leq O_p (E)$. Since $L/R$ is a chief factor of $G$, we get that $R = O_p (E)$, a contradiction. Hence $O_p (E)$ is the unique normal subgroup of $G$ such that $L/O_p (E)$ is a chief factor of $G$. Clearly we have that $O_p (E) \leq Z_{\infty} (E) \cap L \leq Z_{\infty} (L)$. If $L$ is $p$-nilpotent, then $L$ is a $p$-group, which implies that $L \leq O_p (E)$, a contradiction. Therefore we conclude that $L$ is not $p$-nilpotent. Hence there exists a minimal non-$p$-nilpotent subgroup $K$ of $L$. Without loss of generality, we may assume that there exists a Sylow $p$-subgroup of $K$, which is contained in $P$. By {{\cite[Chapter IV, Proposition 5.4]{H1}}} and {{\cite[Chapter VII, Theorem 6.18]{DH}}}, $K$ has the following structure: (1) $K = K_p \rtimes K_q$, where $K_p$ is the normal Sylow $p$-subgroup of $K$, and $K_q$ is a cyclic Sylow $q$-subgroup of $K$; (2) $K_p = K^{\mathfrak{N}}$; (3) $\Phi (K_p) \leq Z(K)$; (4) $K_p / \Phi (K_p)$ is a chief factor of $K$; (5) $exp \,(K_p) =p$ or $4$ (if $p=2$).

Now let $x \in K_p \setminus \Phi (K_p)$, then $o(x) = p$ or $4$ (if $p=2$). Suppose that $[\langle x \rangle,K]=1$, then $\langle x \rangle \leq Z(K)$. Since $\langle x \rangle \not\leq \Phi (K_p)$, we conclude that $\langle x \rangle \Phi (K_p) = K_p$, which yields that $\langle x \rangle = K_p$. By the minimal choice of $p$, we get from {{\cite[Chapter IV, Theorem 2.8]{H1}}} that $K$ is $p$-nilpotent, a contradiction. Hence we get that $[\langle x \rangle, K]>1$. By the normality of $K_p$, we have $[\langle x \rangle,K] \leq K_p$. Suppose that $[\langle x \rangle,K]=K_p$. Since $\langle x \rangle \leq K_p \leq P$, it follows from $\langle x \rangle$ is an $ICPC$-subgroup of $G$ that $\langle x \rangle = \langle x \rangle \cap [\langle x \rangle,K] \leq \langle x \rangle \cap [\langle x \rangle,G] \leq \langle x \rangle_{pcG}$. Therefore we have $\langle x \rangle = \langle x \rangle_{pcG}$. Since $p \,|\,|L/O_p (E)|$, it indicates that $\langle x \rangle$ covers or avoids $L/O_p (E)$. The former one  implies that $L = O_p (E) \langle x \rangle$, which yields that $L$ is a $p$-group, a contradiction. Hence $\langle x \rangle$ avoids $L/O_p (E)$, and so $\langle x \rangle \leq O_p (E)$. By {{\cite[Chapter 1, proposition 1.9]{GW}}}, we get that $\Phi (K) = Z_{\infty} (K)$. Therefore we have that $\langle x \rangle \leq K_p \cap Z_{\infty} (E) \cap K \leq K_p \cap Z_{\infty} (K)$. Clearly $K_p \not\leq \Phi (K)$, thus we have $\Phi (K_p) \leq K_p \cap \Phi (K) < K_p$. Since $K_p /\Phi (K_p)$ is a chief factor of $G$, it indicates that $\Phi (K_p) = K_p \cap \Phi (K)$. Therefore $\langle x \rangle \leq \Phi (K_p)$, a contradiction. Hence we conclude that $1 <[\langle x \rangle,K] <K_p$. 

Suppose that $K_p \cap [\langle x \rangle, G] \leq \Phi (K_p)$, then $[\langle x \rangle, K] \leq \Phi (K_p)$. It follows from (4) that:
$$\langle x \rangle \Phi (K_p)  = \langle x \rangle [\langle x \rangle,K]\Phi (K_p)  = \langle x \rangle ^K\Phi (K_p)  = K_p. $$
Hence we get that $K_p = \langle x \rangle$, which implies again from {{\cite[Chapter 1, proposition 1.9]{GW}}} that $K$ is $p$-nilpotent, a contradiction. Hence we have $1<K_p \cap [\langle x \rangle,G] \not\leq \Phi (K_p)$. Since $K_p \cap [\langle x \rangle,G] \unlhd K$, it follows that $K_p \cap [\langle x \rangle,G] = K_p$, i.e. $K_p \leq [\langle x \rangle,G]$. By our hypothesis, $\langle x \rangle =\langle x \rangle \cap  [\langle x \rangle,G]=\langle x \rangle_{pcG}$. Since $p\,|\,|L/O_p (E)|$, by a similar argument as above, we deduce a contradiction as well. Hence we finally conclude that $O_p (E)=1$.

Let $N$ be a minimal normal subgroup of $G$ such that $N \leq E$. Clearly $N \cap P$ is a Sylow $p$-subgroup of $N$, therefore every cyclic subgroup of $P \cap N$ of order $p$ or $4$ (if $P$ is a non-abelian $2$-group) is an $ICPC$-subgroup of $G$. Hence $(G,N)$ satisfies the hypothesis. If $N <E$, it follows from the minimal choice of $(G,E)$ that $N$ is $p$-nilpotent. It indicates from $O_{p'} (E) =1$ that $N$ is a $p$-group. Thus we conclude that $N \leq O_p (E) =1$, a contradiction. It implies that $N=E$. Now let $\langle x \rangle$ be a cyclic subgroup of $P$ of order $p$ or $4$ (if $P$ is a non-abelian $2$-group). By our hypothesis, $\langle x \rangle$ is an $ICPC$-subgroup of $G$. Hence $\langle x \rangle \cap [\langle x \rangle,G] \leq \langle x \rangle_{pcG}$. Clearly $1 <[\langle x \rangle, G] \unlhd G$, and $\langle x \rangle [\langle x \rangle, G] = \langle x \rangle^G \unlhd E$, therefore we have $[\langle x \rangle,G] = E$. Hence $\langle x \rangle \cap [\langle x \rangle,G]=\langle x \rangle \leq \langle x \rangle_{pcG}$, which yields that $\langle x \rangle =\langle x \rangle_{pcG}$. Since $p\,|\,|E|$, we obtain that $\langle x \rangle$ either covers or avoids $E/1$ as a chief factor of $G$. The former case suggests that $E = \langle x \rangle$, which is absurd by {{\cite[Chapter 1, proposition 1.9]{GW}}}. The latter case suggests that  $E \cap \langle x \rangle = 1$, a contradiction as well. Hence no such counterexample of $(G,E)$ exists, and our proof is complete.
\end{proof}

\end{document}